\documentclass[letterpaper, 10 pt, conference]{ieeeconf}  

\IEEEoverridecommandlockouts                              

\usepackage{graphics} 
\usepackage{epsfig} 
\usepackage{mathptmx} 
\usepackage{times} 
\usepackage{amsmath} 
\usepackage{amssymb}  
\usepackage{algorithm}
\usepackage{algpseudocode} 
\usepackage{subcaption}

\newtheorem{thm}{Theorem}
\newtheorem{lem}{Lemma}
\newtheorem{prop}{Proposition}
\newtheorem{assum}{Assumption}

\newtheorem{remark}{Remark}

\title{\LARGE \bf
A Bayesian Perspective on the Data-Driven LQR
}

\author{Thierry Schwaller, Feiran Zhao, Florian D\"{o}rfler
\thanks{T. Schwaller, F. Zhao, and F. D\"{o}rfler are with the Department of Information Technology and Electrical Engineering, ETH Zürich, 8902 Zürich, Switzerland (e-mail: tschwaller@ethz.ch; zhaofe@control.ee.ethz.ch; dorfler@control.ee.ethz.ch)}%
}

\begin{document}

\maketitle
\thispagestyle{empty}
\pagestyle{empty}

\begin{abstract}
The data-driven linear quadratic regulator (ddLQR) is a widely studied control method for unknown dynamical systems with disturbance. Existing approaches, both indirect, i.e., those that identify a model followed by model-based design, and direct, which bypasses the identification step, often rely on the certainty-equivalence principle and therefore do not explicitly account for model uncertainty. In this paper, we propose a Bayesian formulation for both indirect and direct ddLQR that incorporates posterior uncertainty into the control design. The resulting expected cost decomposes into a certainty-equivalence term and a variance-dependent term, providing a principled interpretation of regularization. We further show that the indirect and direct formulations are equivalent under this perspective. The resulting direct method admits a tractable semidefinite program whose size is independent of the data length. Numerical simulations demonstrate improved optimality gap and closed-loop stability, particularly in low-data regimes.
\end{abstract}

\section{INTRODUCTION}
The linear quadratic regulator (LQR) is the benchmark for validating and comparing different data-driven control methods \cite{anderson2007optimal}, where control policies are learned from data without explicit model knowledge. Existing approaches to data-driven LQR can be broadly classified as indirect methods, where a dynamical model is identified followed by model-based control design, and direct methods, which bypass system identification. Both approaches have their own pros and cons; we refer to \cite{dorfler2023data} for a more detailed discussion.

Existing direct and indirect LQR approaches are predominantly based on the certainty-equivalence principle \cite{mania2019certainty, 8933093, dorfler2022bridging, dorfler2023certainty, tsiamis2023statistical, ivan21, bartos2025stability}. In particular, the indirect approach regards the dynamical model estimated from raw data as the ground truth and solves the corresponding LQR problem. Thus, the uncertainty of indirect LQR design stems from the modeling error induced by noisy data \cite{dorfler2023certainty}. The direct approach parameterizes the state-feedback gain as a linear function of a batch of persistently exciting data, leading to a data-based closed-loop parameterization by directly neglecting the noise \cite{8933093}. As such, the LQR problem can be reformulated as a data-based convex program without involving any explicit SysID. A covariance parameterization of the LQR problem \cite{zhao2025data} is further proposed for more economic computation in \cite{8933093} and to achieve direct adaptive control. The uncertainty of these direct approaches lies in the closed-loop parameterization and is neglected by following certainty equivalence. 

As a result, the uncertainty induced by noise is not explicitly accounted for in control design of both approaches, which usually leads to overconfident or even unstable controllers especially in low signal-to-noise ratio regimes. To compensate the uncertainty, both indirect and direct approaches incorporate regularization \cite{pillonetto2022regularized, dorfler2022bridging, zhao2025regularization}. In particular, the indirect approach imposes a regularizer for system identification to deal with ill-conditioned data and incorporate prior model knowledge from a Bayesian perspective \cite{pillonetto2022regularized}. The direct approach adds a projection-based or $2$-norm regularizer to the LQR cost to single-out a favorable solution and promote robustness~\cite{dorfler2023certainty}. However, the regularization coefficients need ad hoc tuning, and the relations among the regularization methods in direct and indirect approaches remain unclear.





Instead of following the certainty-equivalence principle, recent works \cite{chiuso2025harnessing, baggio2024bayesian} have taken a Bayesian perspective to the data-driven predictive control problem. Specifically, the objective is to find optimal inputs that minimizes the conditional expectation of the finite-horizon linear quadratic cost given past data, named final control error (FCE). Importantly, it was shown that the FCE can be decoupled into a certainty-equivalence cost plus an additional variance-dependent term, where the latter captures the uncertainty in the predictor and acts as a principled regularizer. Moreover, the regularizer coincides with those in previous literature of data-driven prediction control \cite{coulson2019data, dorfler2022bridging}, and the optimal regularization coefficients can be estimated without heuristic tuning. 

Motivated by \cite{chiuso2025harnessing, baggio2024bayesian}, we propose a Bayesian formulation of the data-driven LQR problem, where the objective is the conditional expectation of the infinite-horizon LQR cost given a batch of persistently exciting data and prior model knowledge. By using one-step predicted state approximation, we show that the expected cost can be decomposed into the certainty-equivalence cost plus a variance-dependent term as in \cite{chiuso2025harnessing, baggio2024bayesian}. This decomposition provides a principled interpretation of regularization in both direct and indirect data-driven LQR. 
Our contributions are summarized below.
\begin{itemize}
    \item We formulate the data-driven LQR problem from a Bayesian perspective and show how posterior uncertainty is propagated into the control design.
    \item We derive a variance-based regularization term from the posterior covariance of the model parameters.
    \item We obtain a new covariance-parametrized direct formulation that incorporates prior knowledge and can be cast as a tractable semidefinite programming problem.
\end{itemize} 

We demonstrate through numerical simulations that the Bayesian LQR solution improves the optimality gap and closed-loop robustness compared to certainty-equivalence and covariance-parametrized baselines \cite{zhao2025regularization}.

The remainder of this paper is organized as follows. Section II provides preliminaries on the LQR problem. Section III formulates the data-driven LQR problem from a Bayesian perspective. Section IV proposes the direct and indirect approaches for the Bayesian LQR. Section V shows simulation results. Conclusions are made in Section VI.

\textit{Notations.} We use $I_n$ to denote the $n$-by-$n$ identity matrix. 
    We use $\mathbb{E}[\cdot]$ to denote the expected value.
    We use $\|\cdot\|_F$ to denote the Frobenius norm defined by $\|A\|_F = \sqrt{\operatorname{Tr}\left(AA^\top\right)}$.
    We use $X\in \mathbb{R}^{p\times q} \sim \mathcal{MN}(\bar{X},\Sigma_c, \Sigma_r)$ to denote that $X$ is drawn from a matrix normal distribution, where $\mathbb{E}[\mathrm{vec}(X)] = \mathrm{vec}(\bar{X})$ and $\operatorname{Var}(\mathrm{vec}(X))=\Sigma_c \otimes \Sigma_r$.     
    We use $\bar{X}$ to denote the a priori estimate of the random variable $X$ and $\hat{X}$ to denote the a posteriori estimate of $X$.

\section{Preliminaries: the linear quadratic regulator problem and the data-driven setting}
This section provides the preliminaries on the linear quadratic regulator (LQR) and its formulation as an optimization problem, as well as the certainty-equivalence LQR with Bayesian estimate.

\subsection{The model-based LQR problem}
Consider a linear time-invariant system
\begin{equation*}
    x_{k+1} = Ax_k + Bu_k + w_k,
\end{equation*} where $A\in\mathbb{R}^{n\times n}$, $B\in\mathbb{R}^{n\times m}$, $x_k \in \mathbb{R}^n$, $u_k \in \mathbb{R}^m$, and $w_k \in \mathbb{R}^n$. The infinite horizon LQR aims to find a static state-feedback matrix $K\in \mathbb{R}^{m\times n}$, which minimizes the expected infinite horizon cost with an initial state $x_0$
\begin{equation}\label{eq:lqr_problem_true}
\begin{aligned}
    \min_K \quad & \limsup_{N\rightarrow \infty} \mathbb{E}\left[\frac{1}{N}\sum_{k=0}^{N-1}x_k^\top Q x_k + u_k^\top R u_k\right] \\
    \text{s.t.}\quad & x_{k+1}= Ax_k + Bu_k + w_k, \quad u_k = Kx_k,
\end{aligned}
\end{equation}
where $Q \in \mathbb{R}^{n\times n}\succeq 0$ and $R \in \mathbb{R}^{m\times m} \succ 0$ are penalty matrices. We make the following assumptions on the noise.

\begin{assum}\label{assum_noise}
    The noise sequence $\{w_k\}$ is identically and independently distributed with $w_k \sim \mathcal{N}(0,\sigma_w^2I_n)$, where $\sigma_w^2$ is assumed to be known.
\end{assum}


For all stabilizing $K$ and under \textit{Assumption \ref{assum_noise}}, problem \eqref{eq:lqr_problem_true} can be reformulated as \cite{anderson2007optimal}
\begin{subequations}\label{eq:lqr_opt_problem}
    \begin{align}
    \min_{K,\Sigma \succeq 0} \quad & \operatorname{Tr} \bigl((Q+K^\top R K)\Sigma\bigr) \\
    \text{s.t.}\quad & \Sigma = \sigma_w^2I_n + (A+BK)\Sigma (A+BK)^\top. \label{equ:sigma}
    \end{align}
\end{subequations}
This formulation expresses the LQR cost in terms of the steady-state covariance matrix $\Sigma$ of the closed-loop system, which is the unique positive definite solution to the Lyapunov equation \eqref{equ:sigma}.

When $(A, B)$ are known, the LQR problem can be solved via a Riccati equation. In the sequel, we introduce the certainty-equivalence LQR for the case where $(A, B)$ are unknown.

\subsection{Certainty-equivalence LQR with Bayesian model estimate}
We assume that the system matrices $A$ and $B$ are unknown, but we have access to an offline data set $\mathcal{D}$ and a Gaussian prior of $(A, B)$. 
The data set $\mathcal{D}=\{ X_0, X_1, U_0 \}$ consists of states $X_0$, inputs $U_0$, and successor states $X_1$,  
\begin{align*}
    X_0 &:= [x_0 \ x_1 \ \dots \ x_{T-1}] \in \mathbb{R}^{n \times T}, \\
    U_0 &:= [u_0 \ u_1 \ \dots \ u_{T-1}] \in \mathbb{R}^{m \times T}, \\
    X_1 &:= [x_1 \ x_2 \ \dots \ x_{T}] \in \mathbb{R}^{n \times T}.
\end{align*}
The data set is generated by applying an input sequence $u_k$ to the unknown open-loop system with i.i.d. noise. 
The data sets in $\mathcal{D}$ are connected by the following equation
\begin{equation}\label{eq:data_dynamics}
    X_1 = AX_0 + BU_0 + W_0
    = \begin{bmatrix}
        B & A
    \end{bmatrix}
    \underbrace{\begin{bmatrix}
        U_0 \\ X_0
    \end{bmatrix}}_{D_0}
    + W_0,
\end{equation}
where $W_0$ is the collection of the state noises $w_k$ defined as
\begin{equation*}
    W_0 := [w_0 \ w_1 \ \dots \ w_{T-1}] \in \mathbb{R}^{n \times T}.
\end{equation*}

\begin{assum}[persistency of excitation]\label{assum_pe}
    The input-state data $D_0$ is persistently exciting \cite{willems2005note}, i.e.,
\begin{equation*}
    \mathrm{rank}(D_0) = n+m.
\end{equation*}
\end{assum}

\begin{assum}[prior knowledge]\label{assum_gaussian_mat}
    The system matrices $(A,B)$ are drawn from a matrix normal distribution \cite{gupta2018matrix}
\begin{equation*}
   \begin{bmatrix}
        B & A
    \end{bmatrix}
    \sim
    \mathcal{MN}
    \left(
    \begin{bmatrix}
        \bar{B} & \bar{A}
    \end{bmatrix},
    I_{n},
    \Omega^{-1}
    \right),
\end{equation*}
where $\Omega$ is the precision matrix defined by
\begin{equation*}
    \Omega =
    \begin{bmatrix}
        \Omega_B & 0 \\ 0 & \Omega_A
    \end{bmatrix}
    \succeq 0.
\end{equation*} 
Note that the columns of $A$ and $B$ are uncorrelated, but the rows are correlated by $\Omega_B$ or $\Omega_A$ respectively.
\end{assum}

 Under the Gaussian prior of \textit{Assumption \ref{assum_gaussian_mat}} and the linear Gaussian dynamics \eqref{eq:data_dynamics}, the posterior distribution of the system matrices $p\left(\begin{bmatrix}B & A\end{bmatrix}\mid\mathcal{D}\right) $ is matrix normal.
 Consequently, the closed-loop matrix $A_\text{cl}= A+BK$
is itself a Gaussian random matrix for any fixed feedback gain $K$. This observation allows for a Bayesian interpretation of data-driven LQR. 

According to Appendix \ref{app:reg_least_squares}, the maximum a posteriori estimate of the system matrices is found by solving the following regularized least-squares problem 
\begin{equation}\label{eq:regularized_least_squares}
    \begin{aligned}
        \begin{bmatrix}
        \hat{B} & \hat{A}
        \end{bmatrix}
        = 
        \arg \min_{\begin{bmatrix}B & A\end{bmatrix}} \; &\frac{1}{\sigma_w^2}\|X_1 - \begin{bmatrix}
            B & A
        \end{bmatrix} D_0\|^2_F \\
        &+
        \left\|
        \begin{bmatrix}
        \Delta B & \Delta A
        \end{bmatrix} 
        \Omega^{1/2}
        \right\|^2_F,
    \end{aligned}
\end{equation} 
where $\begin{bmatrix}
        \Delta B &\Delta A
        \end{bmatrix} :=\begin{bmatrix}
        B & A
        \end{bmatrix} 
        -
        \begin{bmatrix}
        \bar{B} & \bar{A}
        \end{bmatrix}$.

\begin{lem}[posterior distribution]\label{lem:post_dist}
    The posterior distribution of $(A,B)$ given $\mathcal{D}$ is a matrix normal
    \begin{equation*}
        \begin{bmatrix}
        B & A
        \end{bmatrix} \mid \mathcal{D} \sim \mathcal{MN}\left(\begin{bmatrix}
        \hat{B} & \hat{A}
        \end{bmatrix},I_n,\Sigma_{B,A}\right),
    \end{equation*}
    where
    the mean is given by the maximum a posteriori estimate
    \begin{equation*}
        \begin{bmatrix}
        \hat{B} & \hat{A}
        \end{bmatrix}
        =
        \left(X_1D_0^\top + \sigma_w^2\begin{bmatrix}
        \bar{B} & \bar{A}
        \end{bmatrix}\Omega \right )/T \, \Psi^{-1}
    \end{equation*}
    and 
    $
        \Sigma_{B,A} = \sigma_w^2/T\Psi^{-1}
    $ is the covariance, where $\Psi$ is the regularized covariance matrix of $D_0$ defined by
    \begin{equation*}
        \Psi = (D_0D_0^\top + \sigma_w^2\Omega)/T \succ 0
    \end{equation*}
\end{lem}
 
The proof is provided in Appendix \ref{app:proof_lemma1} for self-completeness.
 


Following the certainty-equivalence principle, we substitute $(A, B)$ with its posterior expectation from \textit{Lemma \ref{lem:post_dist}} in the LQR problem \eqref{eq:lqr_opt_problem}, leading to the indirect formulation
\begin{equation}\label{eq:ce_bayesian_ddlqr}
\begin{aligned}
    \min_{K,\Sigma \succeq 0} \quad & \operatorname{Tr} \bigl((Q+K^\top R K)\Sigma\bigr) \\
    \text{s.t.}\quad & \Sigma = \sigma_w^2I_n + \begin{bmatrix}
        \hat{B} & \hat{A}
    \end{bmatrix}\begin{bmatrix}
        K \\ I_n
    \end{bmatrix}\Sigma \begin{bmatrix}
        K \\ I_n
    \end{bmatrix}^\top  \begin{bmatrix}
        \hat{B} & \hat{A}
    \end{bmatrix}^\top.
\end{aligned}
\end{equation}
This certainty-equivalence optimization problem ignores the posterior variance and might therefore lead to overconfident or unstable controllers with limited data. This problem will be solved by introducing a regularization term.

\section{A Bayesian approach to the data-driven LQR}
This section presents a Bayesian formulation of the data-driven LQR problem, which enables a systematic incorporation of uncertainty and leads to both indirect and direct Bayesian data-driven LQR formulations. 

\subsection{The Bayesian formulation of the LQR problem}
Consider the Bayesian setting, where we assume $(A,B)$ to be random matrices.
Given a batch of data $\mathcal{D}$, our goal is to find a stabilizing gain $K$ that minimizes the expectation of the LQR cost, where the expectation is taken with respect to the random noise $\{w_t\}$ and random $(A,B)$:
\begin{equation}\label{eq:ddlqr_true_problem}
\begin{aligned}
    \min_K \quad & \limsup_{N\rightarrow \infty} \mathbb{E}\left[\left.\frac{1}{N}\sum_{k=0}^{N-1}\|x_k\|^2_{Q+K^\top R K}\right | \mathcal{D}\right] \\
    \text{s.t.}\quad & x_{k+1}= A_\text{cl}x_k + w_k,
\end{aligned}
\end{equation}
where we denote the closed-loop matrix as $A_\text{cl}:= A+BK$.
To obtain a tractable formulation of \eqref{eq:ddlqr_true_problem}, we first define the nominal state $\bar{x}_k$, which follows the expected dynamics
\begin{equation}\label{equ:nominal}
    \bar{x}_{k+1} = \underbrace{\mathbb{E}\left [ A_\text{cl} \mid \mathcal{D} \right]}_{:=\hat{A}_\text{cl}}\bar{x}_k + w_k,
    \quad
    \bar{x}_0 = x_0.
\end{equation}
Denote $\Delta A_\text{cl} := A_\text{cl} - \hat{A}_\text{cl}$. Then, we define the deviated state as the difference between the state we aim to minimize and the nominal state $e_{k+1} := x_{k+1} - \bar{x}_{k+1}$, which by \eqref{equ:nominal} satisfies the following dynamics
\begin{equation*}
    e_{k+1} = A_\text{cl}e_k + \Delta A_\text{cl} \bar{x}_k, \quad e_0 = 0.
\end{equation*}
In particular, $e_{k+1}$ consists of two terms: the first term is the accumulated state deviation before time $k$, and the second is the one-step prediction error from time $k$. To achieve a tractable approximation of the posterior expected cost, we retain only the one-step predicted error 
\begin{equation}\label{eq:one_step_error_approx}
   e_{k+1} \approx \Delta A_\text{cl}  \bar{x}_k.
\end{equation}
As we will see later, with this approximation, we actually neglect the moments of the uncertainty higher than second-order ones in the expected cost function.


Then, the state square norm can be approximated by three terms, i.e.,
\begin{equation}
\begin{aligned}
    \|x_k\|^2 &= \|\bar{x}_k + e_k\|^2  = \|\bar{x}_k\|^2 + \|e_k\|^2 + 2\bar{x}_k^{\top}e_k \\
    & \overset{\eqref{eq:one_step_error_approx}}{\approx} \|\bar{x}_k\|^2 + \|\Delta A_\text{cl}  \bar{x}_{k-1}\|^2 + 2\bar{x}_k^{\top}\Delta A_\text{cl}  \bar{x}_{k-1}.
\end{aligned}
\end{equation}
In particular, the conditional expectation of the last term is zero, i.e.
\begin{equation*}
\begin{aligned}
      & \mathbb{E}_{\{w_t\},A,B}\left[\left.\bar{x}_k^\top \Delta A_\text{cl} \bar{x}_{k-1}\right|\mathcal{D}\right] \\
      &= \mathbb{E}_{\{w_t\}}\left [\left.\bar{x}_k^\top \mathbb{E}_{A,B}[\Delta A_\text{cl}|\mathcal{D}] \bar{x}_{k-1}\right|\mathcal{D}\right] = 0,
\end{aligned}
\end{equation*}
where the last equality follows from $\mathbb{E}_{A,B}[\Delta A_\text{cl}|\mathcal{D}]=0$.
Furthermore, the LQR problem in \eqref{eq:ddlqr_true_problem} becomes 
\begin{subequations}\label{eq:bayesian_main_approx_problem}
\begin{align}
    \min_K   & \lim_{N\rightarrow \infty} \mathbb{E}\left[\left.\frac{1}{N}\sum_{k=0}^{N-1}\left(\|\bar{x}_k\|^2_{Q+K^\top R K} + \|\Delta A_\text{cl}\bar{x}_k\|^2_{Q+K^\top R K}\right)\right | \mathcal{D}\right] \label{equ:sepa_cost} \\
    \text{s.t.} ~ & \bar{x}_{k+1}= \hat{A}_\text{cl}\bar{x}_k + w_k, ~\bar{x}_0=x_0
\label{eq:bayesian_main_approx_problem_2}
\end{align}
\end{subequations}
The expected posterior cost is therefore composed of a nominal cost and the prediction mismatch due to parameter uncertainty, following the nominal trajectory \eqref{eq:bayesian_main_approx_problem_2}.


\begin{remark}[separation principle]
    The separation of the expected posterior cost into the nominal certainty-equivalence cost plus the predicted state variance cost in \eqref{equ:sepa_cost} is closely related to the separation principle in data-driven predictive control \cite{chiuso2025harnessing}, where the uncertainty enters the cost also through the variance of the predicted outputs. 
\end{remark}

Next, we derive indirect and direct data-driven LQR (ddLQR) methods based on the Bayesian formulation of the LQR problem \eqref{eq:bayesian_main_approx_problem}.
\subsection{The indirect Bayesian LQR}
In the general case, the term added to the nominal system is given by
\begin{equation*}
    \frac{1}{N}\sum_{k=0}^{N-1}\mathbb{E}\left[\left.\|\Delta A_\text{cl}\bar{x}_k\|^2_{Q+K^\top R K}\right | \mathcal{D}\right].
\end{equation*}
To simplify the presentation and obtain a tractable regularization term, the penalty matrix of the state norm is chosen approximately as
\begin{equation}\label{eq:norm_approx}
    Q+K^\top R K \approx \frac{1}{n}\operatorname{Tr}\left(Q+K^\top R K\right)I_n =: \frac{\lambda_0}{n} I_n,
\end{equation}
where $\lambda_0$ is some constant to be tuned.

Since
\begin{equation}\label{eq:error_term_expectation}
\begin{aligned}
&\mathbb{E}_{\{w_t\},A,B}\left[\|\Delta A_\text{cl}\bar{x}_k\|^2\mid \mathcal{D}\right] \\
    &=
    \mathbb{E}_{\{w_t\}}\left[\left.\bar{x}_k^\top \mathbb{E}_{A,B}\left [\Delta A_\text{cl}^\top \Delta A_\text{cl}|\mathcal{D}\right] \bar{x}_k \right | \mathcal{D}\right],    
\end{aligned}
\end{equation}
and the posterior of $\begin{bmatrix}\Delta B & \Delta A\end{bmatrix}$ is matrix normal distributed $\begin{bmatrix}\Delta B & \Delta A\end{bmatrix}|\mathcal{D} \sim \mathcal{MN}\left(0_{n\times n},I_n,\Sigma_{B,A}\right)$, the inner expectation value can be written as
\begin{equation*}
    \mathbb{E}_{A,B}\left [\left.\Delta A_\text{cl}^\top \Delta A_\text{cl} \right|\mathcal{D}\right] =  n\cdot\begin{bmatrix}
        K^\top & I_n
    \end{bmatrix}\Sigma_{B, A}\begin{bmatrix}
        K \\ I_n
    \end{bmatrix},
\end{equation*}
where $\Sigma_{B,A}$ is the variance of the model estimate in Lemma \ref{lem:post_dist}.
Taking the expectation over $\{w_k\}$ of \eqref{eq:error_term_expectation} results in
\begin{equation}\label{eq:regularization_term_general}
     \mathbb{E}_{\{w_t\}}\left[\|\Delta A_\text{cl}\bar{x}_k\|^2\mid \mathcal{D}\right] 
    =
    n\operatorname{Tr}\left(\begin{bmatrix}
        K^\top & I_n
    \end{bmatrix}\Sigma_{B, A}\begin{bmatrix}
        K \\ I_n
    \end{bmatrix}\Sigma\right),
\end{equation}
where $\Sigma$ is the solution to the Lyapunov equation \eqref{equ:sigma}.
Then, the Bayesian LQR formulation \eqref{eq:ce_bayesian_ddlqr} becomes modified to
\begin{equation}\label{eq:ce_bayesian_ddlqr_regularized}
\begin{aligned}
    \min_{K,\Sigma \succeq 0} \quad &
    \operatorname{Tr} \bigl((Q+K^\top R K)\Sigma\bigr)
    + \lambda_0 \operatorname{Tr}\left(\begin{bmatrix}
        K^\top & I_n
    \end{bmatrix}\Sigma_{B, A}\begin{bmatrix}
        K \\ I_n
    \end{bmatrix}\Sigma\right) \\
    \text{s.t.}\quad &
    \Sigma = \sigma_w^2I_n + \hat{A}_\text{cl} \Sigma \hat{A}_\text{cl}^\top .
\end{aligned}
\end{equation}

By the definition of $\Sigma_{B,A}$ in \textit{Lemma \ref{lem:post_dist}}, the regularization term shown in \eqref{eq:regularization_term_general} can  be written as
\begin{equation}\label{equ:reg}
    \lambda\operatorname{Tr}
    \left (
    \begin{bmatrix}
        K^\top & I_n
    \end{bmatrix} 
    \Psi^{-1} 
    \begin{bmatrix}
        K \\ I_n
    \end{bmatrix} 
    \Sigma
    \right),
\end{equation}
where $\lambda := \sigma_w^2/T \lambda_0$ is a tunable hyperparameter, which should be chosen $\lambda \propto 1/T$.
\begin{thm}[indirect Bayesian data-driven LQR]
    Consider the stochastic system introduced in \eqref{eq:ddlqr_true_problem} under \textit{Assumptions \ref{assum_noise}-\ref{assum_gaussian_mat}} and the posterior distribution in \textit{Lemma \ref{lem:post_dist}}. Under the one-step approximation $e_{k+1}=\Delta A_\text{cl}\bar{x}_k$ and using the approximation in \eqref{eq:norm_approx}, the posterior expected infinite-horizon LQR cost admits the tractable formulation
\begin{equation}\label{eq:indirect_bayesian_ddlqr}
\begin{aligned}
    \min_{K,\Sigma \succeq 0} \quad &
    \operatorname{Tr} \bigl((Q+K^\top R K)\Sigma\bigr)
    + \lambda\operatorname{Tr}
    \left (
    \begin{bmatrix}
        K^\top & I_n
    \end{bmatrix} 
    \Psi^{-1} 
    \begin{bmatrix}
        K \\ I_n
    \end{bmatrix} 
    \Sigma
    \right) \\
    \text{s.t.}\quad &
    \Sigma = \sigma_w^2I_n + 
    \begin{bmatrix}
        \hat{B} & \hat{A}
    \end{bmatrix}\begin{bmatrix}
        K \\ I_n
    \end{bmatrix}\Sigma \begin{bmatrix}
        K \\ I_n
    \end{bmatrix}^\top\begin{bmatrix}
        \hat{B} & \hat{A}
    \end{bmatrix}^\top,
\end{aligned}
\end{equation}
where $\begin{bmatrix}
        \hat{B} & \hat{A}
    \end{bmatrix}$ 
    is given by \eqref{eq:regularized_least_squares} and $\lambda = \sigma_w^2\lambda_0/T$.
\end{thm}
\begin{proof}
    The result follows from the decomposition $x_k=\bar{x}_k+e_k$, the one-step approximation, the approximation in \eqref{eq:norm_approx}, and the posterior covariance expression in \textit{Lemma \ref{lem:post_dist}}.
\end{proof}

\begin{remark}[exploitation effect of the regularization]
The regularized cost function of \eqref{eq:ce_bayesian_ddlqr_regularized} can be written as
\begin{equation*}
    \operatorname{Tr} \left(
    \begin{bmatrix}
        K \\ I_n
    \end{bmatrix}^\top\left(
    \begin{bmatrix}
        R & 0 \\ 0 & Q
    \end{bmatrix}
    + \lambda 
    \Psi^{-1} 
    \right)\begin{bmatrix}
        K \\ I_n
    \end{bmatrix}\Sigma\right).
\end{equation*} 
The regularization modifies the quadratic cost matrices by adding a data-dependent term proportional to the posterior covariance of the system parameters. The eigenvectors of $\Psi^{-1}$ corresponding to large eigenvalues correspond to directions in the parameter space with high posterior uncertainty. Therefore, the controller tends to take safe actions along well-explored parameter directions, which is referred to as exploitation in the context of reinforcement learning. This is closely related to the exploitation and exploration effect of a similar regularization method \cite{zhao2025regularization}.
\end{remark}
\subsection{The direct Bayesian LQR}
We can now cast the indirect data-driven LQR presented in \eqref{eq:indirect_bayesian_ddlqr} into a direct ddLQR, which bypasses the system identification. 
Under \textit{Assumptions \ref{assum_pe}} and  \ref{assum_gaussian_mat}, the regularized empirical covariance matrix $\Psi$ from \textit{Lemma \ref{lem:post_dist}} is positive definite, and there is a unique solution $V\in\mathbb{R}^{(n+m)\times n}$ to 
\begin{equation}\label{eq:covariance_parameterization}
    \begin{bmatrix}
        K \\ I_n
    \end{bmatrix} = \Psi V 
     =: \begin{bmatrix}
        \Psi_1 \\ \Psi_2
    \end{bmatrix}V
\end{equation}
for any given $K$.
Furthermore, we define the first $m$ rows of $\Psi$ as $\Psi_1 \in \mathbb{R}^{m\times (n+m)}$ and the other $n$ rows as $\Psi_2 \in \mathbb{R}^{n\times (n+m)}$.

With this parameterization we can reformulate the posterior distribution of the closed-loop $A_\text{cl}$ given $\mathcal{D}$ according to \textit{Lemma \ref{lem:post_dist}} as
\begin{equation*}
    A_\text{cl} | \mathcal{D} \sim \mathcal{MN}\left(\begin{bmatrix}
        \hat{B} & \hat{A}
    \end{bmatrix}\Psi V,I_n,\sigma_w^2/T V^\top \Psi V\right),
\end{equation*}
where the expected value of the closed loop system given $\mathcal{D}$ is 
\begin{equation*}
    \begin{bmatrix}
        \hat{B} & \hat{A}
    \end{bmatrix}\Psi V
    = \left(X_1D_0^\top + \sigma_w^2\begin{bmatrix}
        \bar{B} & \bar{A}
    \end{bmatrix}\Omega\right ) V/T =:\bar{X}_1V.
\end{equation*}
By setting $K=\Psi_1 V$ and enforcing $I_n = \Psi_2 V$, we can enforce \eqref{eq:covariance_parameterization} in the optimization problem \eqref{eq:indirect_bayesian_ddlqr}, which leads to
\begin{equation}\label{eq:direct_bayesian_ddlqr}
\begin{aligned}
    \min_{K,\Sigma \succeq 0} \quad &
    \operatorname{Tr} \bigl((Q+V^\top \Psi_1^\top R \Psi_1 V)\Sigma\bigr)
    + \lambda\operatorname{Tr}(\Psi V \Sigma V^\top) \\
    \text{s.t.}\quad &
    \Sigma = \sigma_w^2I_n + \bar{X}_1V \Sigma V^\top \bar{X}_1^\top,~~I_n = \Psi_2 V.
\end{aligned}
\end{equation}
The next result shows that the proposed indirect and direct formulations are equivalent, so the direct method preserves the same control objective while avoiding explicit model identification.
\begin{thm}
The direct data-driven approach presented in \eqref{eq:direct_bayesian_ddlqr} and the indirect data-driven approach presented in \eqref{eq:indirect_bayesian_ddlqr} are equivalent in the sense that their solutions coincide.
\end{thm}
\begin{proof}
Since $\Psi$ is positive definite, the variable $V$ solving \eqref{eq:covariance_parameterization} is uniquely determined as $V=\Psi^{-1}\begin{bmatrix}
    K^\top & I_n
\end{bmatrix}^\top$ we can reformulate \eqref{eq:direct_bayesian_ddlqr} as
\begin{equation*}
\begin{aligned}
    \min_{K,\Sigma \succeq 0} \quad &
    \operatorname{Tr} \bigl((Q+K^\top R K)\Sigma\bigr) + \lambda\operatorname{Tr}\left(\begin{bmatrix}
        K \\ I_n
    \end{bmatrix}^\top\Psi^{-1} \begin{bmatrix}
        K \\ I_n
    \end{bmatrix} \Sigma \right) \\
    \text{s.t.}\quad &
    \Sigma = \sigma_w^2I_n + \hat{A}_\text{cl}\Sigma \hat{A}_\text{cl}^\top.
\end{aligned}
\end{equation*}
Furthermore, we know that $\hat{A}_\text{cl} = \mathbb{E}\!\left[ A_\text{cl} \,\middle|\, \mathcal{D}\right]$ which, according to \textit{Lemma \ref{lem:post_dist}}  coincides with the definition of the solution to the regularized least-squares problem \eqref{eq:regularized_least_squares}.
Therefore, we conclude that the two formulations \eqref{eq:indirect_bayesian_ddlqr} and \eqref{eq:direct_bayesian_ddlqr} are equivalent.
\end{proof}

Next, we show that \eqref{eq:direct_bayesian_ddlqr} admits a semi-definite program (SDP) formulation, where the size of the optimization variables is independent of the data size $T$.  
\begin{prop}\label{prop_sdp}
    If \eqref{eq:direct_bayesian_ddlqr} is feasible, then the optimal gain for \eqref{eq:direct_bayesian_ddlqr} can be computed by $K = \Psi_1 S \Sigma^{-1}$, where $S$ and $\Sigma$ are given by
    \begin{equation}\label{eq:sdp_direct_bayesian_ddlqr}
    \begin{aligned}
        \min_{\Sigma,S,L,M}\quad & \operatorname{Tr} (Q \Sigma) + \operatorname{Tr} (R L) + \lambda\operatorname{Tr}(M\Psi) \\
        \text{s.t.}\quad & \Psi_2S = \Sigma, 
        \quad \begin{bmatrix}
            \Sigma - \sigma_w^2I_n & \bar{X}_1 S \\
            S^\top \bar{X}_1^\top,\quad\Sigma
        \end{bmatrix} \succeq 0, \\
        & \begin{bmatrix}
            L & \Psi_1 S \\
            S^\top \Psi_1^\top & \Sigma
        \end{bmatrix} \succeq 0,  
         \begin{bmatrix}
            M & S \\
            S^\top & \Sigma
        \end{bmatrix} \succeq 0.
    \end{aligned}
    \end{equation}
\end{prop}

\begin{proof}
We can reformulate \eqref{eq:direct_bayesian_ddlqr} by using the change of variables $V=S\Sigma^{-1}$ as
\begin{equation*}
\begin{aligned}
    \min_{L,M,S,\Sigma \succeq 0} \quad &
    \operatorname{Tr} (Q\Sigma)+\operatorname{Tr}(LR) + \lambda\operatorname{Tr}( M \Psi) \\
    \text{s.t.}\quad &
    \Psi_2 S = \Sigma, \\
    & \Sigma \succeq \sigma_w^2I_n\\
    &\Sigma \succeq \bar{X}_1 S\Sigma^{-1}S^\top\bar{X}_1^\top + \sigma_w^2I_n, \\
    & L \succeq \Psi_1 S \Sigma^{-1} S^\top \Psi_1^\top, \\
    & M \succeq S\Sigma^{-1}S^\top.
\end{aligned}
\end{equation*}
Using the Schur complement for the last four inequalities in the constraints, the proof is completed.
\end{proof}

The dimensions of the matrices in \eqref{eq:sdp_direct_bayesian_ddlqr} do not depend on the data length and can therefore be efficiently solved by modern SDP solvers (e.g., \cite{goulart2024clarabel}). Therefore, by following Algorithm \ref{alg:direct_bayes_short} we can find the state-feedback gain $K$.
\begin{remark}[non-informative prior]
    The covariance-parametrization shown in \eqref{eq:covariance_parameterization} was used similarly in \cite{zhao2025regularization}, however instead of the regularized covariance matrix $\Psi$ the covariance matrix $\Phi$ was used
    \begin{equation*}
        \Phi = D_0D_0^\top/T. 
    \end{equation*}
    In the case of the non-informative prior, meaning $\Omega$ is equal to the zero matrix, the covariance parametrization in \cite{zhao2025regularization}  and the direct Bayesian LQR coincide. 
\end{remark}
\begin{algorithm}[t]
\caption{Direct Bayesian LQR}
\label{alg:direct_bayes_short}
\begin{algorithmic}[1]
\Require Data $\mathcal{D}=\{X_0,X_1,U_0\}$, weights $(Q,R)$, prior $\left(\begin{bmatrix}\hat{B} & \hat{A}\end{bmatrix},\Omega\right)$, $\sigma_w^2$
\Ensure State-feedback gain $K$

\State Form $D_0 = \begin{bmatrix} U_0 \\ X_0 \end{bmatrix}$
\State Compute regularized data covariance $\Psi = \frac{1}{T}(D_0D_0^\top+\sigma_w^2\Omega)$
\State Solve the SDP in \textit{Proposition \ref{prop_sdp}} to obtain $(\Sigma,S)$
\State Recover controller $K = \Psi_1 S \Sigma^{-1}$
\State \Return $K$
\end{algorithmic}
\end{algorithm}

\section{Simulation results}
In this section, simulation results for the Bayesian direct data-driven LQR will be shown using a discrete-time second order spring-mass-damper system as an example.
We consider a discrete-time second-order spring-mass damper system with state
$x_k = [\,p_k \;\; v_k\,]^\top$ (position and velocity)
\begin{equation*}
x_{k+1} = A x_k + B u_k + w_k,
\qquad
w_k \sim \mathcal N(0,\sigma_w^2 I_2),
\end{equation*}
where the system matrices are 
\begin{equation*}
A =
\begin{bmatrix}
1 & T_s\\
-\alpha & 1-\beta
\end{bmatrix},
\qquad
B =
\begin{bmatrix}
0\\
\gamma
\end{bmatrix}.
\end{equation*}
For a small sampling time $T_s$, the coefficients admit the physical
interpretation $\alpha \approx (k/m)T_s$, 
$\beta \approx (c/m)T_s$, and 
$\gamma \approx (1/m)T_s$. We set $T_s = 1$.
However, in this work we treat $(\alpha,\beta,\gamma)$ directly as
discrete-time parameters.

We place independent Gaussian priors on the parameters
\begin{equation*}
\alpha \sim \mathcal N(\bar{\alpha},\sigma_\alpha^2),
\qquad
\beta \sim \mathcal N(\bar{\beta},\sigma_\beta^2),
\qquad
\gamma \sim \mathcal N(\bar{\gamma},\sigma_\gamma^2).
\end{equation*}
This induces a matrix-normal prior on
$(A,B)$ of the form
\begin{equation*}
\begin{bmatrix}B & A\end{bmatrix} \sim \mathcal{MN}(\begin{bmatrix}\bar{B} & \bar{A}\end{bmatrix}, I_2, \Omega_0^{-1}),
\end{equation*}
with prior mean
\begin{equation*}
\begin{bmatrix}\bar{B} & \bar{A}\end{bmatrix} =
\begin{bmatrix}
0 & 1 & T_s\\
\bar{\gamma} & -\bar{\alpha} & 1-\bar{\beta}
\end{bmatrix},
\end{equation*}
and row-wise precision matrix
\begin{equation*}
\Omega_0 =
\mathrm{diag}\!\left(
\sigma_\gamma^{-2},
\sigma_\alpha^{-2},
\sigma_\beta^{-2}
\right).
\end{equation*}
In each run $\alpha$, $\beta$, and $\gamma$ were randomly sampled. To keep the data collection step marginally stable, the spectral radius of $A$ was restricted to be at most $1.05$. 

We will focus on two metrics, namely
\paragraph{Empirical optimality gap}
The empirical optimality gap is defined as 
\begin{equation*}
    \mathcal{E} = \frac{C(K)-C^*}{C^*}.
\end{equation*}
It compares the cost accumulated by applying the controller found compared to the theoretical optimal cost $C^*$. 
This optimality gap is calculated in each run, and the median value of all these values is considered if the found $K$ stabilizes the system. 

\paragraph{Stability rate}
The stability rate is defined as the percentage of runs in which a stabilizing controller was determined by applying the found controller $K$.

The data set $\mathcal{D}$ is collected by randomly sampling $U_0$, $W_0$, and $x_0$ from Gaussian distributions and applying them to the open-loop system to obtain $X_0$ and $X_1$. For each data point, 10,000 runs were made.

Two types of simulations were considered; one explores the effect that regularization $\lambda$ has on the above-mentioned metrics and the other explores the effects of collecting more data $T$ has on the metrics. In both simulations, we compare the covariance-parametrized approach with the proposed approach.

The parameters in the simulation are shown in
Table $\ref{table_sim_params}$.

\begin{table}[h]
\caption{Simulation Parameters}
\label{table_sim_params}
\begin{center}
\begin{tabular}{|c|c|c|c|c|c|c|c|c|}
\hline
  $\bar{\alpha}$ & $\bar{\beta}$ & $\bar{\gamma}$ & $\sigma_{\alpha}$ & $\sigma_{\beta}$ & $\sigma_{\gamma}$& $\sigma_w$ &Q & R \\
\hline
 $1.05$ & $0.05$ & $1$ & $0.5$ & $0.1$ & $0.8$ & $0.25$ & $\operatorname{diag}(5,0.1)$ & $0.1$ \\
\hline
\end{tabular}
\end{center}
\end{table}

\subsubsection{Effects of the regularization}
To assess the effects of regularization, we ran the simulation for different values of $\lambda$, while the data length was chosen to be $T=8$. 
The results of the simulations are displayed in Figure \ref{fig:discrete_spring_mass_results}. 

\begin{figure}
  \centering
  \subfloat[a][Stability rate of the Bayesian LQR and of the covariance-parametrized LQR for different values of $\lambda$]{\includegraphics[width=0.75\linewidth]{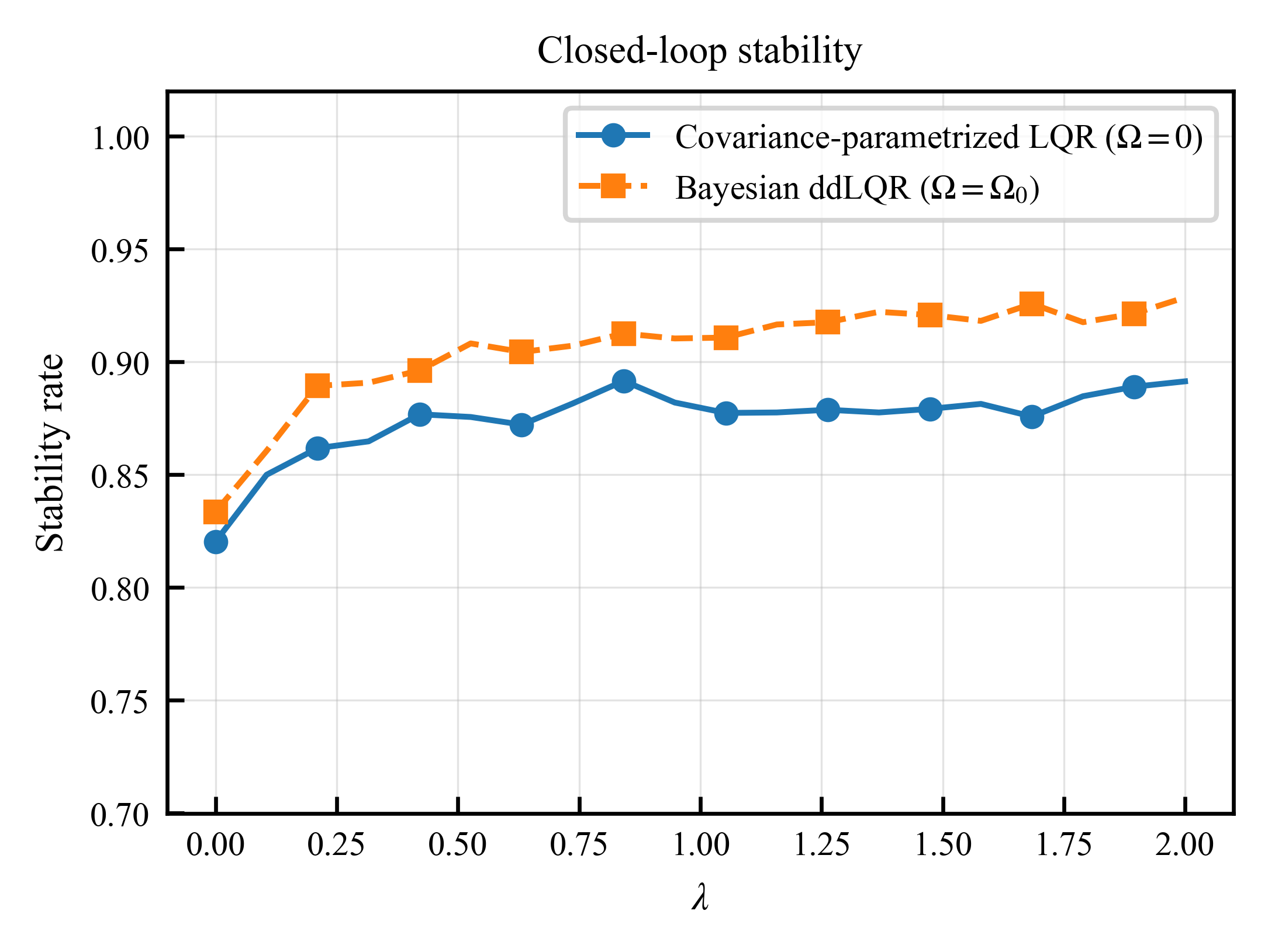} \label{fig:a}} \\
  \subfloat[b][Median optimality gap of the Bayesian LQR and of the covariance-parametrized LQR for different values of $\lambda$]{\includegraphics[width=0.75\linewidth]{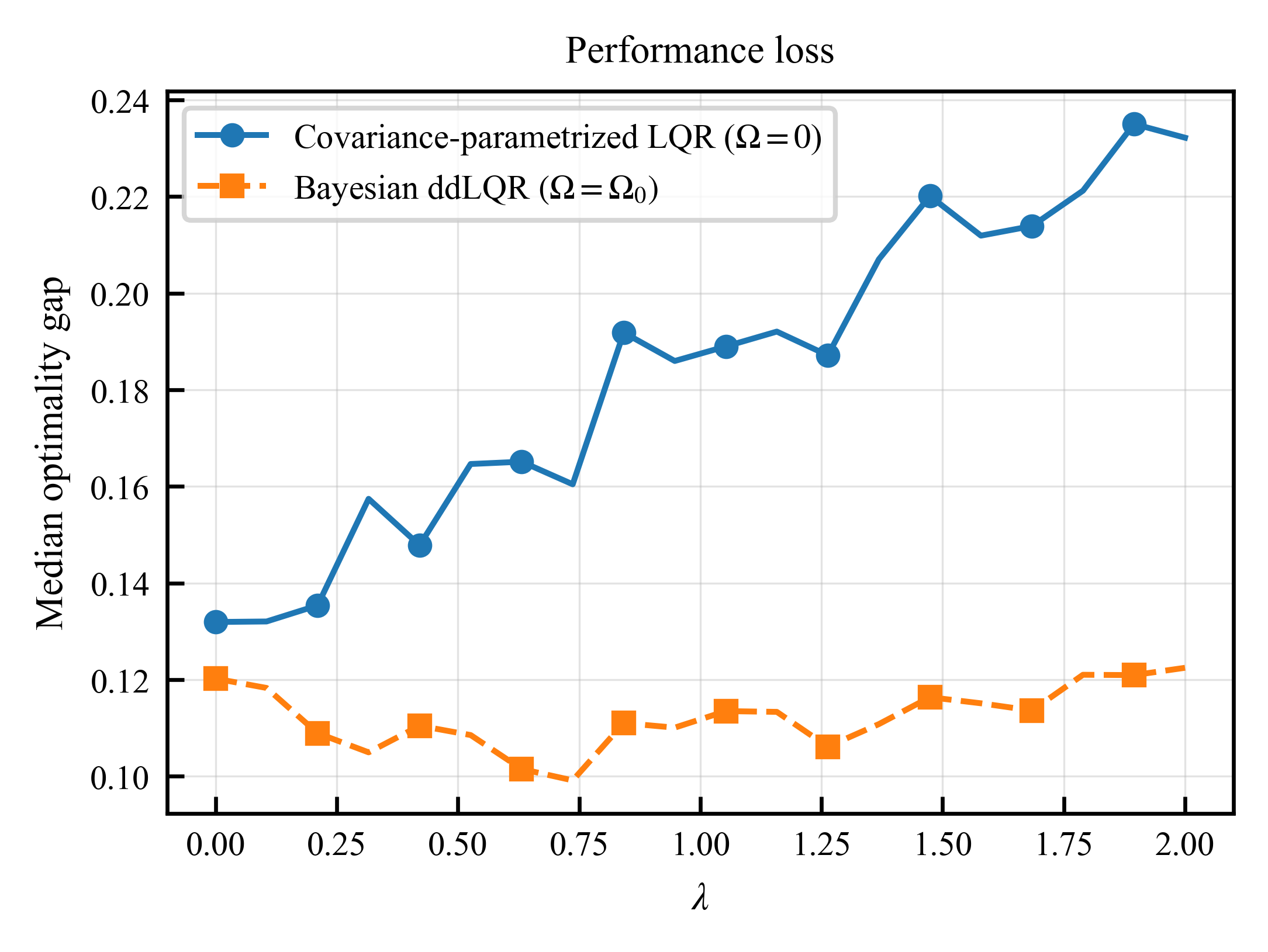} \label{fig:b}}
  \caption{Effects of the regularization for the covariance-parametrized LQR and for the Bayesian ddLQR.} \label{fig:discrete_spring_mass_results}
\end{figure}

As $\lambda$ increases, the stability rate for both approaches improves to a point. However, as seen especially for the covariance-based approach, choosing $\lambda$ too large can result in worse performance. These results support the interpretation of the posterior covariance term as a robustness-promoting penalty.

\subsubsection{Effects of data size}
To assess the effect of data size, we performed the simulation for different values of $T$, while the hyperparameter $\lambda$ was chosen as $\lambda=1/T$.

\begin{figure}
  \centering
  \subfloat[a][Stability rate comparison between the covariance-parametrized LQR and the Bayesian LQR]{\includegraphics[width=0.75\linewidth]{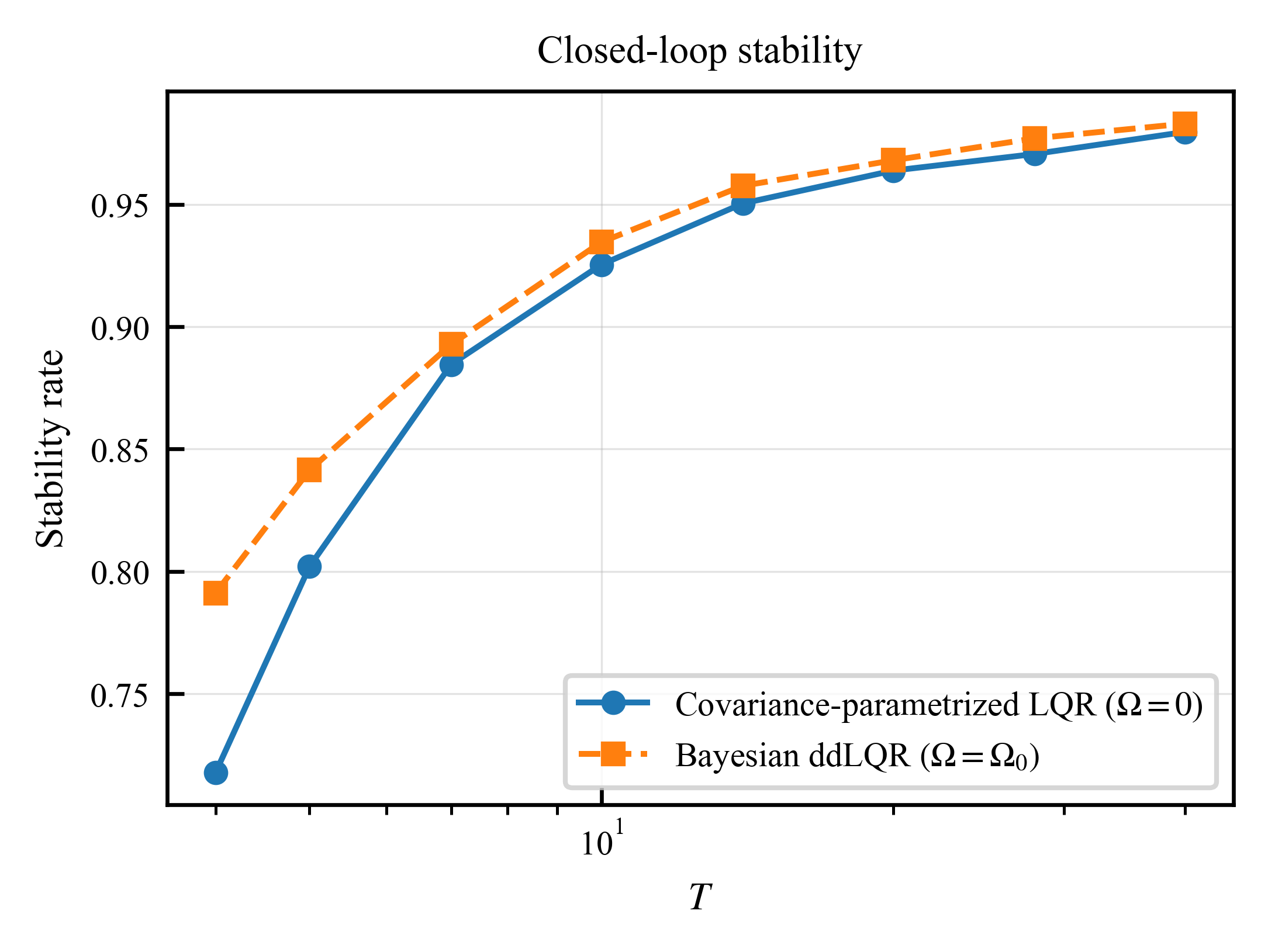} \label{fig:as}} \\
  \subfloat[b][Median optimality gap comparison between the covariance-parametrized LQR and the Bayesian LQR for data sets of length $T$]{\includegraphics[width=0.75\linewidth]{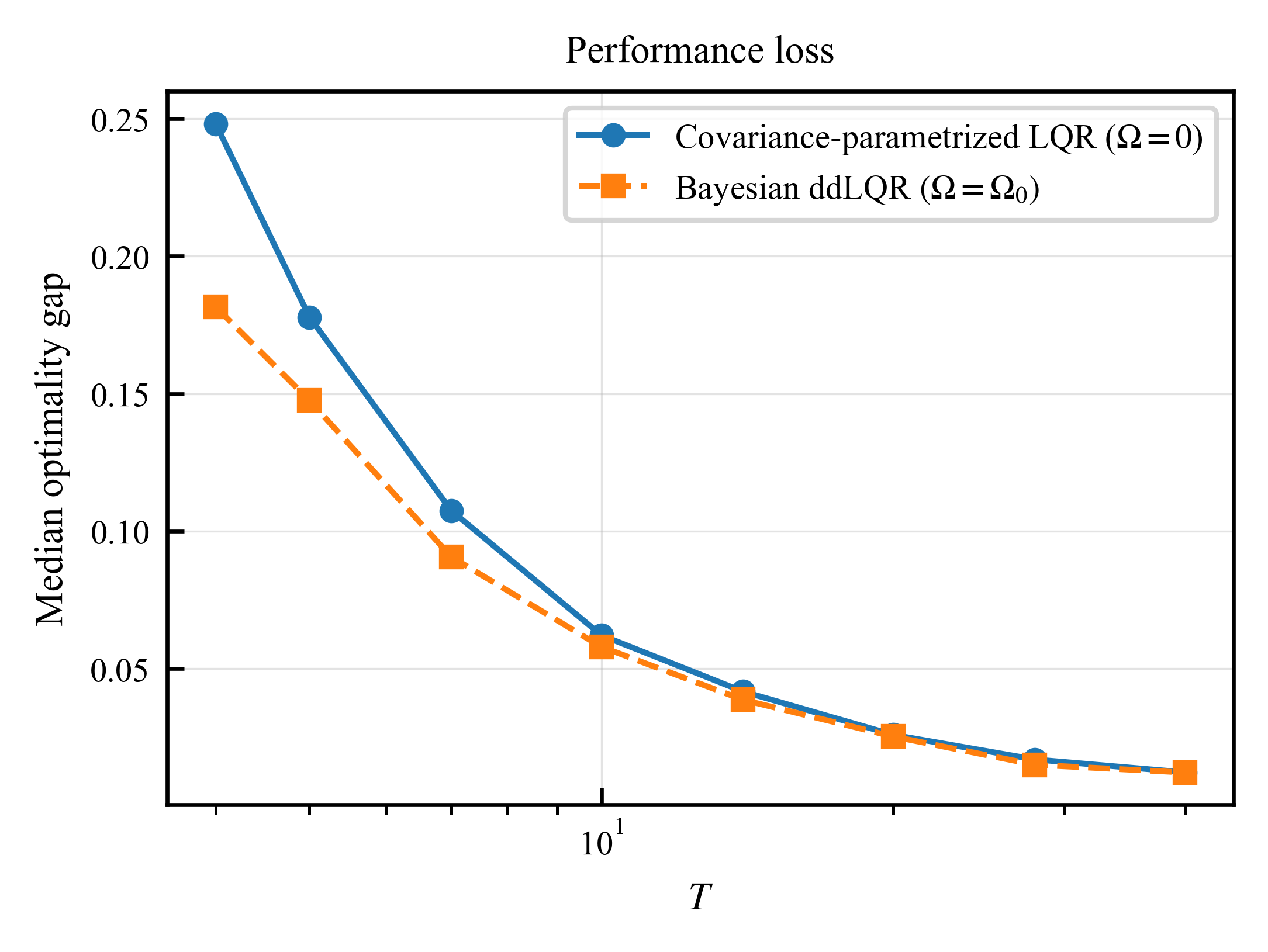} \label{fig:bd}}
  \caption{Effects of the data size for the covariance-parametrized LQR and the Bayesian ddLQR for data sets of length $T$.} \label{fig:discrete_spring_mass_results_2}
\end{figure}

If $T$ is small, the difference in stability and performance is significant in favor of the proposed approach. However, as $T$ increases and the posterior uncertainty shrinks, the two approaches converge in both metrics. These results indicate that the proposed Bayesian regularization is particularly beneficial in low-data regimes, where posterior uncertainty is largest.

\section{Conclusion}
This paper proposed a Bayesian formulation of direct and indirect ddLQR in which posterior uncertainty is propagated into the control design through a variance-based regularization term.
The resulting direct formulation admits an SDP formulation with dimensions independent of the data length, allowing it to be solved efficiently. Simulations indicate that the Bayesian ddLQR achieves a lower median optimality gap and higher closed-loop stability than existing approaches, particularly when limited data are available.
An interesting direction for future work is the extension to an adaptive or online setting as in \cite{zhao2025data}.
\addtolength{\textheight}{-3cm}   



\section*{APPENDIX}
\subsection{Derivation of the regularized least-squares problem}\label{app:reg_least_squares}
Our goal is to find the maximum a posteriori estimate of $(A,B)$ given $\mathcal{D}$. For ease of notation we define
\begin{equation}
     \Theta:=\begin{bmatrix}B & A\end{bmatrix} .
\end{equation}
\begin{equation}\label{eq:appendix_map_est}
    \begin{aligned}
        \Theta^* &= \arg \max_\Theta p(X_1| \Theta,D_0)p(\Theta|D_0)\\
        &= \arg \min_\Theta -\log\left(p(X_1| \Theta,D_0)\right) - \log\left(p(\Theta|D_0)\right)
    \end{aligned} 
\end{equation}
We know that the random variables $\Theta$, $X_1$, and $D_0$ satisfy the following relation
\begin{equation*}
    X_1 = \Theta D_0 + W_0,
\end{equation*}
where, according to \cite{gupta2018matrix} and \textit{Assumption \ref{assum_gaussian_mat}}
\begin{equation}\label{eq:appendix_x1_apost}
    p(X_1|\Theta,D_0) \propto \exp\left(\operatorname{Tr}\left(-\frac{1}{2\sigma_w^2}(X_1-\Theta D_0)(X_1-\Theta D_0)^\top\right)\right)
\end{equation}
and
\begin{equation}\label{eq:appendix_theta}
    p\left(\Theta | D_0\right) \propto \exp \left(\operatorname{Tr}\left(-\frac{1}{2}(\Theta - \bar{\Theta})\Omega(\Theta - \bar{\Theta})^\top\right)\right).
\end{equation}
Therefore, \eqref{eq:appendix_map_est} can be written using the Frobenius norm as
\begin{equation*}
    \Theta^* = 
    \arg \min_{\Theta }\frac{1}{\sigma_w^2}\|X_1 - \Theta D_0\|^2_F 
    +
    \left\|\left(\Theta -\bar{\Theta}\right)\Omega^{1/2}\right\|^2_F,
\end{equation*}
which coincides with \eqref{eq:regularized_least_squares}.
\subsection{Proof of \textit{Lemma 1}}\label{app:proof_lemma1}
For ease of notation we again define
\begin{equation}
    \begin{bmatrix}B & A\end{bmatrix} = \Theta.
\end{equation}
According to \eqref{eq:appendix_map_est}, \eqref{eq:appendix_x1_apost}, and \eqref{eq:appendix_theta} we can write the posterior of $\Theta$
\begin{equation*}
    p(\Theta|\mathcal{D}) \propto p(X_1|\Theta,D_0)p(\Theta).
\end{equation*}
Collecting quadratic and linear terms in $\Theta$ and completing the squares yields
\begin{equation*}
    p(\Theta | \mathcal{D}) \propto \exp\left(\operatorname{Tr}\left(-\frac{1}{2}\left(\Theta-\hat{\Theta}\right)\left(\sigma_w^2/T \Psi^{-1} \right)^{-1}\left(\Theta - \hat{\Theta}\right)^\top\right)\right),
\end{equation*}
where $\hat{\Theta}$ is the posterior mean.
Therefore, according to \cite{gupta2018matrix} the posterior distribution of $\Theta$ given $\mathcal{D}$ is 
\begin{equation*}
    p(\Theta | \mathcal{D}) = \mathcal{MN}\left(\hat{\Theta},I_n,\sigma_w^2/T \Psi^{-1}\right),
\end{equation*}
which concludes the proof.

\bibliographystyle{IEEEtran}
\bibliography{bibfile}

@article{mania2019certainty,
  title={Certainty equivalence is efficient for linear quadratic control},
  author={Mania, Horia and Tu, Stephen and Recht, Benjamin},
  journal={Advances in neural information processing systems},
  volume={32},
  year={2019}
}

@article{bartos2025stability,
  title={Stability of Certainty-Equivalent Adaptive LQR for Linear Systems with Unknown Time-Varying Parameters},
  author={Bartos, Marcell and K{\"o}hler, Johannes and D{\"o}rfler, Florian and Zeilinger, Melanie N},
  journal={arXiv preprint arXiv:2511.08236},
  year={2025}
}

@article{ivan21,
  title={Behavioral systems theory in data-driven analysis, signal processing, and control},
  author={Markovsky, Ivan and D{\"o}rfler, Florian},
  journal={Annual Reviews in Control},
  volume={52},
  pages={42--64},
  year={2021},
  publisher={Elsevier}
}

@book{pillonetto2022regularized,
  title={Regularized system identification-Learning dynamic models from data},
  author={Pillonetto, Gianluigi and Chen, Tianshi and Chiuso, Alessandro and De Nicolao, Giuseppe and Ljung, Lennart and others},
  year={2022},
  publisher={Springer}
}

@article{baggio2024bayesian,
  title={The bayesian separation principle for data-driven control},
  author={Baggio, Giacomo and Carli, Ruggero and Grimaldi, Riccardo Alessandro and Pillonetto, Gianluigi},
  journal={arXiv preprint arXiv:2409.16717},
  year={2024}
}

@inproceedings{coulson2019data,
  title={Data-enabled predictive control: In the shallows of the DeePC},
  author={Coulson, Jeremy and Lygeros, John and D{\"o}rfler, Florian},
  booktitle={2019 18th European control conference (ECC)},
  pages={307--312},
  year={2019},
  organization={IEEE}
}

@article{dorfler2023data,
  title={Data-driven control: Part two of two: Hot take: Why not go with models?},
  author={D{\"o}rfler, Florian},
  journal={IEEE Control Systems Magazine},
  volume={43},
  number={6},
  pages={27--31},
  year={2023},
  publisher={IEEE}
}

@article{tsiamis2023statistical,
  title={Statistical learning theory for control: A finite-sample perspective},
  author={Tsiamis, Anastasios and Ziemann, Ingvar and Matni, Nikolai and Pappas, George J},
  journal={IEEE Control Systems Magazine},
  volume={43},
  number={6},
  pages={67--97},
  year={2023},
  publisher={IEEE}
}

@article{zhao2025regularization,
  title={Regularization for covariance parameterization of direct data-driven LQR control},
  author={Zhao, Feiran and Chiuso, Alessandro and D{\"o}rfler, Florian},
  journal={IEEE Control Systems Letters},
  year={2025},
  publisher={IEEE}
}

@article{dorfler2023certainty,
  title={On the certainty-equivalence approach to direct data-driven LQR design},
  author={D{\"o}rfler, Florian and Tesi, Pietro and De Persis, Claudio},
  journal={IEEE Transactions on Automatic Control},
  volume={68},
  number={12},
  pages={7989--7996},
  year={2023},
  publisher={IEEE}
}

@article{dorfler2022bridging,
  title={Bridging direct and indirect data-driven control formulations via regularizations and relaxations},
  author={D{\"o}rfler, Florian and Coulson, Jeremy and Markovsky, Ivan},
  journal={IEEE Transactions on Automatic Control},
  volume={68},
  number={2},
  pages={883--897},
  year={2022},
  publisher={IEEE}
}

@book{anderson2007optimal,
  title={Optimal control: linear quadratic methods},
  author={Anderson, Brian DO and Moore, John B},
  year={2007},
  publisher={Courier Corporation}
}

@article{willems2005note,
  title={A note on persistency of excitation},
  author={Willems, Jan C and Rapisarda, Paolo and Markovsky, Ivan and De Moor, Bart LM},
  journal={Systems \& Control Letters},
  volume={54},
  number={4},
  pages={325--329},
  year={2005},
  publisher={Elsevier}
}

@article{goulart2024clarabel,
  title={Clarabel: An interior-point solver for conic programs with quadratic objectives},
  author={Goulart, Paul J and Chen, Yuwen},
  journal={arXiv preprint arXiv:2405.12762},
  year={2024}
}

@book{gupta2018matrix,
  title={Matrix variate distributions},
  author={Gupta, Arjun K and Nagar, Daya K},
  year={2018},
  publisher={Chapman and Hall/CRC}
}

@article{chiuso2025harnessing,
  title={Harnessing uncertainty for a separation principle in direct data-driven predictive control},
  author={Chiuso, Alessandro and Fabris, Marco and Breschi, Valentina and Formentin, Simone},
  journal={Automatica},
  volume={173},
  pages={112070},
  year={2025},
  publisher={Elsevier}
}

@article{zhao2025data,
  title={Data-enabled policy optimization for direct adaptive learning of the LQR},
  author={Zhao, Feiran and D{\"o}rfler, Florian and Chiuso, Alessandro and You, Keyou},
  journal={IEEE Transactions on Automatic Control},
  year={2025},
  publisher={IEEE}
}

@ARTICLE{8933093,
  author={De Persis, Claudio and Tesi, Pietro},
  journal={IEEE Transactions on Automatic Control}, 
  title={Formulas for Data-Driven Control: Stabilization, Optimality, and Robustness}, 
  year={2020},
  volume={65},
  number={3},
  pages={909-924},
  doi={10.1109/TAC.2019.2959924}}

\end{document}